\documentclass[reqno,a4paper,12pt]{amsart}

\usepackage{amsmath,amssymb,amsthm,amscd,amsfonts,amsbsy,bm}
\usepackage[mathscr]{eucal}

\font\sc=rsfs10 at 12pt

\numberwithin{equation}{section}

\renewcommand{\a}{\alpha}
\renewcommand{\b}{\beta}

\newcommand{\D}{\Delta}

\newcommand{\ve}{\varepsilon}

\renewcommand{\k}{\kappa}

\newcommand{\m}{\mu}
\newcommand{\n}{\nu}

\renewcommand{\r}{\rho}
\newcommand{\s}{\sigma}
\newcommand{\Si}{\Sigma}

\renewcommand{\t}{\tau}
\newcommand{\f}{\phi}
\newcommand{\vf}{\varphi}

\renewcommand{\P}{\Pi}
\renewcommand{\o}{\omega}
\renewcommand{\O}{\Omega}

\newcommand{\C}{{\mathbb C}}
\newcommand{\R}{{\mathbb R}}

\newcommand{\DD}{{\mathbb{D}}}

\newcommand{\pp}{\pmb{\partial}}

\newcommand{\fb}{{\mathbf f}}

\newcommand{\qb}{{\mathbf q}}

\newcommand{\tb}{{\mathbf t}}

\newcommand{\Ab}{{\mathbf A}}

\newcommand{\Bb}{{\mathbf B}}

\newcommand{\Db}{{\mathbf D}}
\newcommand{\Eb}{{\mathbf E}}

\newcommand{\Hb}{{\mathbf H}}

\newcommand{\Kb}{{\mathbf K}}

\newcommand{\Mb}{{\mathbf M}}

\newcommand{\Pb}{{\mathbf P}}
\newcommand{\Qb}{{\mathbf Q}}

\newcommand{\Sbb}{{\mathbf S}}
\newcommand{\Tb}{{\mathbf T}}

\newcommand{\Vb}{{\mathbf V}}

\newcommand{\Zb}{{\mathbf Z}}

\newcommand{\Ac}{{\mathcal A}}

\newcommand{\Dc}{{\mathcal D}}

\newcommand{\Fc}{{\mathcal F}}

\newcommand{\Hc}{{\mathcal H}}

\newcommand{\Kc}{{\mathcal K}}
\newcommand{\Lc}{{\mathcal L}}

\newcommand{\Sc}{{\mathcal S}}

\newcommand{\Es}{\sc\mbox{E}\hspace{1.0pt}}

\newcommand{\Ls}{\sc\mbox{L}\hspace{1.0pt}}

\DeclareMathOperator{\im}{{\rm Im}\,}

\newtheorem{theorem}{Theorem}[section]
\newtheorem{proposition}[theorem]{Proposition}

\newtheorem*{theorem*}{Theorem}

\theoremstyle{definition}
\newtheorem{definition}[theorem]{Definition}

\theoremstyle{remark}
\newtheorem{remark}[theorem]{Remark}

\date{}

\begin{document}

\title[Poly-Bergman Toeplitz operators]{Toeplitz operators in polyanalytic Bergman type spaces}

\author{Grigori Rozenblum }

\address{ Chalmers University of Technology and The University of Gothenburg (Sweden, Gothenburg); St.Petersburg State University Dept.Math. Physics (St.Petersburg, Russia)}

\email{grigori@chalmers.se}
\author {Nikolai Vasilevski}
\address{Cinvestav (Mexico, Mexico-city)}
\email{nvasilev@cinvestav.mx}
\dedicatory{In memory of Selim Grigorievich Krein, a great mathematician and a charming person}
\begin{abstract}

We consider Toeplitz operators in Bergman and Fock type spaces of polyanalytic $L^2\textup{-}$functions on the disk or on the half-plane with respect to the Lebesgue measure (resp., on $\mathbb{C}$ with  the plane Gaussian measure).  The structure involving creation and annihilation operators, similar to the classical one present for the Landau Hamiltonian, enables us to reduce Toeplitz operators in true polyanalytic spaces to the ones in the usual Bergman type spaces, however with distributional symbols. This reduction leads to describing a number of properties of the operators in the title, which may differ  from the properties of the usual Bergman-Toeplitz operators.

\end{abstract}
\maketitle
\section{Introduction} The paper is devoted to the study of Toeplitz operators in a relatively unexplored class of Bergman type spaces. While there is a vast literature devoted to operators in the classical Bergman and Fock spaces of analytic functions, considerably less is known about the case of spaces of polyanalytic functions. As it concerns general properties of these spaces, one can mention the books \cite{Balk,VBook} and the review papers \cite{Balk2,Feucht}. More specific information concerning  properties of these spaces can be found in \cite{Hedenmalm, Karlovich, Pavlovic, Pessoa15, Pessoa17, V1,  V2, Ramazanov1999, Ramazanov2002}, and some more, see the bibliography below. However, only one  paper \cite{CucLe} was devoted to the study of Toeplitz operators in nonanalytic spaces. Meanwhile, the application fields of polyanalytic spaces keep expanding, including classical and quantum physics, mechanics, signal processing, wavelets and more.

As it has been  rather recently discovered, the spaces of polyanalytic functions possess an interesting structure, reminding the creation-annihilation operators in certain classical quantum-physics problems. Say, for the Landau Hamiltonian, i.e.,  the operator describing the motion of a charged quantum particle confined to the plane under the action of  the uniform magnetic field orthogonal to the plane, the whole $L^2$ space splits into the sequence of orthogonal (Landau) subspaces $\Lc_q$, $q=0,1,\dots$, invariant for the operator. These subspaces are interconnected by creation, $Q^+$, and annihilation, $Q$, operators, $Q^+:\Lc_q\to\Lc_{q+1}$, $Q:\Lc_q\to\Lc_{q-1}$, which are isometries, up to a numerical coefficient, while $Q:\Lc_0\to\{0\}$. These operators, discovered in 1928 by V.Fock \cite{Fock}, play an important role in the analysis of various problems concerning two- and three-dimensional quantum models with magnetic fields. The fact important for us is that the creation and annihilation operators enable one to reduce a Toeplitz operator in the Landau subspace $\Lc_q$, $q>0,$ to the unitary equivalent Toeplitz operator in the lowest subspace $\Lc_0$, however, with a different symbol. Such reduction was first discovered in \cite{BPR}. A few years later, in \cite{RT1,RT2}, such equivalence, (more exactly, 'almost equivalence', up to a very week error term) has been established for an approximate creation-annihilation structure arising for the Landau Hamiltonian with weakly (and not that weakly) perturbed uniform magnetic field.

It turns out that the Bergman type spaces of polyanalytic functions admit such creation-annihilation structure as well. The ambient Hilbert space $L^2(\O)$, for a domain $\O\subset\C$, splits into the direct orthogonal sum of 'true polyanalytic' subspaces $\Ac_{(j)}(\O):=\Ac_j(\O)\ominus\Ac_{j-1}(\O),$ where $\Ac_j(\O)$ consists of polyanalytic functions, $\Ac_j(\O)=\{f\in L^2, \bar{\partial}^jf=0\},$ and, for the case of the disk or the half-plane, also of analogous spaces of anti-polyanalytic functions. This property was established by Ramazanov in \cite{Ramazanov1999,Ramazanov2002} and, in a different  way, by Y.Karlovich in \cite{Karlovich} for the Bergman space  on the unit disk, by N.Vasilevski, in \cite{V2} and Y.Karlovich--L.Pessoa in \cite{KarlPessIEOT}, for the Bergman space  on the upper half-plane and by N.Vasilevski in \cite{V1} for the Fock space. In these papers, it was found that the creation-annihilation operators, acting as partial isometries,  can be represented as two-dimensional singular integral operators.

In the present paper we follow the pattern of the analysis of the (perturbed) Landau Hamiltonian and establish reduction theorems which associate with a Toeplitz operator in a polyanalytic Bergman type space another Toeplitz operator, but now acting in the corresponding usual analytic Bergman type space. The symbol of the new operator is obtained from the symbol of the initial operator by applying an elliptic differential operator. Thus, if the initial symbol was not sufficiently smooth, the reduced symbol may turn out to be not a function but a distribution. The new operator turns out to be unitarily equivalent (or, for the case of the Bergman spaces on the disk, cosimilar, see Definition \ref{DefCosimilar}) to the initial one.

Thus, an alternative arises. Should one, for further analysis, consider operators with nice symbols in 'bad' spaces (the polyanalitic spaces are 'bad' due to the absence of many useful properties and structures present for analytic spaces), or, in the opposite, consider operators in nice spaces with 'bad' symbols. The second option seems to be more productive.
 Fortunately, the analysis of Toeplitz operators with distributional (and, more generally, singular) symbols is now available. For operators in the Bergman space in the disk and in the Fock space, such analysis has been performed by the authors in \cite{RV1,RV2}; the case of the Bergman space in the half-plane will be presented in \cite{RV4}.  Using this analysis, we find conditions for boundedness and compactness of Toeplitz operators in true polyanalytic Bergman type spaces, which turn out to be similar to the conditions in the standard Bergman spaces, and discuss the questions of uniqueness of the symbol and the finite rank problem, which differ somewhat from the classical ones.

 The results are presented for the case of Bergman type spaces on domains in the complex plane. For higher dimensions the approach works as well, with certain modifications. These results will be presented elsewhere later.

\bigskip

\section{The structure of poly-Bergman and poly-Fock spaces}
\subsection{Polyanalytic and true polyanalytic spaces}
A Bergman type Hilbert  space  $\mathbf{B}^2(\O)$ is the Hilbert space of solutions of some elliptic equation or system in a domain $\Omega$ in $\mathbb{R}^d$ or $\mathbb{C}^d$, belonging to $L^2(\Omega)$ with respect to some measure $\m$ (Banach spaces, involving $L^p$ theory, are also considered in the literature, but we do not discuss them here). Specific for Bergman type spaces is the existence of the reproducing kernel $\kappa(z,w)$: for any $z\in\Omega$, the linear evaluation functional $\pmb{ev}_z:\mathbf{B}^2(\O)\ni u\mapsto u(z)$ is continuous and therefore admits the representation $\pmb{ev}_z(u)=\langle u(\cdot), \kappa(z,\cdot)\rangle$. A classical Toeplitz operator  in $\mathbf{B}^2(\O)$ with symbol $F\in L^\infty(\Omega)$ is the  operator acting as $\mathbf{T}_F\equiv \mathbf{T}_F(\mathbf{B}^2(\O)) :\mathbf{B}^2(\O)\ni u\mapsto \Pb F u\in\mathbf{B}^2(\O)$, where $\Pb$ is the orthogonal (Bergman) projection $\Pb:L^2(\Omega)\to \mathbf{B}^2(\O)$. In \cite{RV1}, \cite{RV2} an approach has been developed for defining Toeplitz operators in Bergman type spaces with much more general symbols, including $F$ being a distribution in $\Omega$ or $\Omega\times\Omega$ or even a hyper-function in $\Omega$. This approach is based upon defining operators by means of sesquilinear forms with a special structure and further using the reproducing kernel to describe the action of the operator.

The typical and best studied examples here are the classical Bergman space $\mathcal{A}^2(\mathbb{D})$ of   analytic functions on the unit disk $\mathbb{D}$, a similar space $\Ac^2(\Pi)$ of analytic functions on the upper half-plane $\Pi$, both considered as subspaces of $L^2$ with respect to the corresponding Lebesgue area measure $dA(z)$, and the Fock (Bargmann-Segal) space $\mathcal{F}^2$ of entire analytic functions on $\mathbb{C}\equiv \mathbb{R}^2$, belonging to $L^2$ with respect to the plane Gaussian measure $d\m=d\rho \equiv \pi^{-1}e^{-|z|^2}dA(z)$. Since we deal with Hilbert spaces only, we suppress further on the superscript $2$ in the notation of  spaces  and write simply $\Ac(\DD)$, $\Ac(\Pi) $,  $\Fc$, and so on.

In the present paper we study Toeplitz operators in analogous Hilbert spaces of polyanalytic functions, i.e., spaces of solutions of the iterated Cauchy-Riemann equation $\overline{\partial}^j u=0,$\quad $j=2,3,\dots$, where, as usual, $\overline{\partial}=\frac12(\partial_x+i\partial_y)$. The polyanalytic Bergman spaces (or, shorter, poly-Bergman spaces) $\mathcal{A}_j(\mathbb{D})$, $\mathcal{A}_j({\Pi})$, are the spaces of such functions in $L^2$ in the disk, resp., half-plane, and the poly-Fock space $\mathcal{F}_j$ is the space of entire  $j-$analytic functions in $L^2(\C,d\r)$. On properties of these spaces and applications, see the monographs \cite{Balk}, \cite{VBook}, the review articles \cite{Balk2}, \cite{Feucht} and references therein (note that in the monograph \cite{VBook}, for the case of the spaces  on the disk, the poly-Bergman spaces are defined in a slightly different way, as functions satisfying  $(\overline{z}\bar{\partial})^j u=0$; our results essentially carry over to this case as well). Polyanalytic functions find now applications in various problems in quantum and classical physics and mechanics, signal processing, wavelet theory etc.

 Of course,
$\mathcal{A}_{j-1}({\O})\subset \mathcal{A}_j({\O}),$
 therefore, to exclude poly-analytic functions of lower order, we define \emph{true} poly-Bergman and \emph{true} poly-Fock
spaces as
$${\mathcal{A}}_{(j)}(\mathbb{D})=\mathcal{A}_j(\mathbb{D})\ominus\mathcal{A}_{j-1}(\mathbb{D}), \,{\mathcal{A}}_{(j)}(\Pi)=\mathcal{A}_j(\Pi)\ominus\mathcal{A}_{j-1}(\Pi), {\mathcal{F}}_{(j)}=\mathcal{F}_j\ominus\mathcal{F}_{j-1}.$$
These spaces have been introduced and studied in \cite{Ramazanov1999,V1,V2}.
Along with poly-Bergman type spaces $\mathcal{A}_j$, one can introduce anti-poly-Bergman type spaces $\tilde{\mathcal{A}}_j$, i.e., spaces of solutions of the iterated adjoint Cauchy-Riemann equation $\partial^j u=0$, $\partial=\frac12(\partial_x-i\partial_y)$ and the corresponding true anti-poly-Bergman type spaces $\tilde{\mathcal{A}}_{(j)}$. Here, the properties for Bergman spaces and Fock spaces differ essentially. It is established (see \cite{V1}) that true poly-Fock spaces form a complete orthogonal system in $L^2(\C,\rho)$ in the sense that $$\overline{\bigcup_{j=1}^\infty \mathcal{F}_j}=\bigoplus_{j=1}^\infty \Fc_{(j)}=L^2(\C,d\r).$$ On the other hand, any true poly-Bergman space on the upper half-plane  is orthogonal to  any true anti-poly-Bergman space;  and therefore
$$\left(\bigoplus_{j=1}^\infty {\mathcal{A}}_{(j)}(\Pi)\right)\oplus\left(\bigoplus_{j=1}^\infty \tilde{\mathcal{A}}_{(j)}(\Pi)\right)=L^2(\Pi,dA),  $$
with a similar relation for the unit disk.
Since the iterated Cauchy-Riemann operator is elliptic, all the above spaces are reproducing kernel ones.

\subsection{The structure of the Fock spaces.}\label{subsec.Fock} Along with true poly-Fock spaces $\Fc_{(j)}$, we consider the Landau subspaces $\Lc_q\in L^2(\C,dA)$, $q=0,1,\dots$, the eigenspaces of the magnetic Schr\"odinger operator $\Hb=-(\nabla + i\Ab(z))^2$, where $\Ab=\frac12(y,-x)$, is the potential of the uniform magnetic field, $z=x+iy$. As it was found by V.Fock \cite{Fock} in 1928, the eigenspace $\Lc_0$ consists of functions of the form $g=\s u\in L^2(\C,dA)$, $\s(z)=\pi^{-\frac12}e^{-|z|^2/2}=\o(z)^{\frac12}$, with $u$ being an entire analytic function on $\C$. Thus the unitary mapping $\Vb:L^2(\C,d\r)\to L^2(\C,dA)$, $\Vb: u(z)\mapsto \s(z)u(z)$ maps isometrically the Fock space $\Fc=\Fc_{(1)}$ onto $\Lc_0$ . Further on, the 'creation operator' $\Qb^+=(-\partial+\frac12 \bar{z})$ maps the Landau subspace $\Lc_q$ onto $\Lc_{q+1}$, isometrically, up to a constant factor depending on $q$, while the 'annihilation operator' $\Qb=\bar{\partial}+\frac12 z$ maps $\Lc_q$ onto $\Lc_{q-1}$, $q>0,$ $\Qb\Lc_0=\{0\}$. Under the unitary transformation $\Vb$, the operators $\Qb^+,\Qb$ are transformed to the operators
\begin{equation}\label{CreAnniFock}
\Sbb^+=\Vb^{-1} \Qb^+\Vb=-\partial+\bar{z}, \quad \Sbb=\Vb^{-1} \Qb\Vb=\bar{\partial}.
\end{equation}

Note that for the space $\Vb^{-1}\Lc_{q-1},$
\begin{equation*}
 \bar{\partial}^q(\Vb^{-1}\Lc_{q-1}) = \Sbb^q(\Vb^{-1}\Lc_{q-1}) = \Vb(\Qb^q\Lc_{q-1})) = 0,
\end{equation*}
thus the space $\Vb^{-1}\Lc_{q-1} = \Vb^{-1}((\Qb^+)^{q-1}\Lc_0) = (\Sbb^+)^{q-1}\Fc_{(1)} = (\Sbb^+)^{q-1} \Fc$ consists of $q$-analytic functions. On the other hand, since the subspaces $\Lc_{q-1}=\Qb^{q-1}\Lc_0$ are mutually orthogonal, the subspace $(\Sbb^+)^{q-1} \Fc=\Vb^{-1} (\Qb^+)^{q-1}\Lc_0$ is orthogonal to the previous subspaces $(\Sbb^+)^l\Fc$, $l=0,\dots,q-2$. Therefore,
\begin{equation}\label{SF}
 (\Sbb^+)^{q-1}\Fc=\Fc_{(q)}
\end{equation}
is the space of true  $q$-analytic functions, and the creation and annihilation operators $\Sbb^+$ and $\Sbb$ connect isometrically, up to multiplicative constants, the true poly-Fock spaces $\Fc_{(j)}$.

Another approach to these properties, based upon the Fourier representation of the true poly-Fock spaces is elaborated in \cite{V1}, where, in particular, an exact isometry between these spaces is described.
\begin{theorem}\label{Fockexact}
 Given natural numbers $k < n$, the operator
\begin{equation*}
 \sqrt{\frac{(k-1)!}{(n-1)!}} (\Sbb^+)^{n-k}|_{\Fc_{(k)}} : \ \Fc_{(k)} \ \longrightarrow \ \Fc_{(n)}
\end{equation*}
is an isometric isomorphism, together with its inverse
\begin{equation*}
 \sqrt{\frac{(k-1)!}{(n-1)!}} \Sbb^{n-k}|_{\Fc_{(n)}} : \ \Fc_{(n)} \ \longrightarrow \ \Fc_{(k)}
\end{equation*}
\end{theorem}


\subsection{The structure of the Bergman spaces on the half-plane and on the disk.} Although the Bergman spaces on the disk and on the half-plane are related by the Caley transform, the properties of the corresponding true poly-Bergman  spaces differ somewhat.

For the case of the Bergman space on the upper half-plane, with the \emph{standard orthonormal basis} $\sqrt{\frac{k+1}{\pi}}\frac{(z-i)^{k}}{(z+i)^{k+1}},\, k=0,1,\dots$ there exists a system of creation and annihilation operators, described in \cite{V2,V3, KarlPessIEOT}. These operators are two-dimensional singular integral operators,
\begin{equation*} 
(\Sbb_\P u)(w)=-\frac1\pi \int_{\P}\frac{u(z)dA(z)}{(z-w)^2} \qquad \text{and} \qquad (\Sbb_\P^*u)(w)=-\frac1\pi \int_{\P}\frac{u(z)dA(z)}{(\bar{z}-\bar{w})^2}.
\end{equation*}
Understood in the principal value sense, they are bounded in $L^2(\Pi)$ and adjoint to each other. They are, in fact,  the Beurling--Ahlfors operators compressed to the half-plane, and  are surjective isometries,
\begin{equation}\label{SBP}
  \Sbb_\P:  \Ac_{(j)}(\P)\to \Ac_{(j+1)}(\P), \quad \Sbb_\P:  \tilde{\Ac}_{(j)}(\P)\to \tilde{\Ac}_{(j-1)}(\P), \quad j>1,
\end{equation}
and
\begin{equation}\label{S*BP}
 \Sbb_\P^*:  \Ac_{(j+1)}(\P)\to \Ac_{(j)}(\P), \quad \Sbb_\P^*:  \tilde{\Ac}_{(j-1)}(\P)\to \tilde{\Ac}_{(j)}(\P), \quad j>1,
\end{equation}
while
\begin{equation*}
    \Sbb_\P^*:\Ac_1(\P)\to \{0\}, \quad \Sbb_\P : \tilde{\Ac}_1(\P)\to \{0\}.
\end{equation*}

Thus, we have surjective isometries
\begin{equation*}
  (\Sbb_\P)^j\Ac(\P) \equiv (\Sbb_\P)^j\Ac_{(1)}(\P)=\Ac_{(j+1)}(\P)
\end{equation*}
and
\begin{equation*}
  (\Sbb_\P^*)^j\tilde{\Ac}(\P) \equiv (\Sbb_\P^*)^j\tilde{\Ac}_{(1)}(\P)=\tilde{\Ac}_{(j+1)}(\P).
\end{equation*}
Formulas \eqref{SBP}, \eqref{S*BP}, similar to \eqref{SF},  justify calling  $\Sbb_\P,\Sbb^*_\P$ creation and annihilation operators.

\begin{remark} Here one can notice a certain discrepancy in notations: $\Sbb$ denotes the annihilation operator in the poly-Fock spaces while $\Sbb_\P$ denotes the creation operator in the poly-Bergman spaces -- however, this is the tradition we do not want to break.
\end{remark}

The operators $\Sbb_\P^j$, restricted to $\Ac(\P)$, admit a representation, found in \cite{PessoaPW}, which is much more convenient for using in further reductions.
\begin{theorem}[{\cite[Theorem 3.3]{PessoaPW}}] For $u\in\Ac(\P)$,
\begin{equation}\label{TransP}
   ( \Sbb_\P^j u)(z)=\frac{{\partial}^{j}[(z-\bar{z})^{j}u(z)]}{j!},  j\ge0.
\end{equation}
\end{theorem}Now we pass on to the case of the Bergman space on the unit disk $\DD$, with the \emph{standard orthonormal basis} $\sqrt{\frac{k+1}{\pi}}z^k, \, k=0,1,\dots$. Similar to the case of the half plane we define the operators
\begin{equation*} 
    (\Sbb_\DD u)(w)=-\frac1\pi \int_{\DD}\frac{u(z)dA(z)}{(z-w)^2} \qquad \text{and} \qquad (\Sbb_\DD^* u)(w)=-\frac1\pi \int_{\DD}\frac{u(z)dA(z)}{(\bar{z}-\bar{w})^2}.
\end{equation*}
These operator are however not isometries of true poly-Bergman spaces but only partial isometries.
Let $\Ls_j$ denote the one-dimensional space $\{\t \bar{z}^{j-1}, \t\in \C\}$, $\widetilde{\Ls}_j=\{\t {z}^{j-1}, \t\in \C\}$.
As shown in \cite[Theorem~3.5]{Karlovich}, the operator $S_\DD$ acts surjective isometrically:
\begin{equation*}
    \Sbb_\DD: \Ac_{(j)}(\DD)\ominus\Ls_j\to\Ac_{(j+1)}(\DD),
\end{equation*}
while   $\Sbb_\DD\Ls_j=\{0\}$ (with a natural modification for $\Sbb_\DD^*$.)

Our aim now is to present differential operators which can replace  $\Sbb_\DD^j$,  similarly to \eqref{TransP}. We apply the results of A. Ramazanov \cite[Theorem 1]{Ramazanov1999}, see also \cite[Theorem 1, Corollary 2]{Ramazanov2002}, and \cite[Theorem~4.1]{Pessoa15}. As it has been proved there (in the notations of the present paper),
the operator
\begin{equation}\label{SDD}
    \Si_j: \Ac_{(1)}(\DD)\ni u\mapsto \partial^{j}((1-|z|^2)^{j} u)
\end{equation}
is a bounded and boundedly invertible operator from $\Ac_{(1)}(\DD)\equiv\Ac(\DD)$ onto $\Ac_{(j+1)}(\DD)$. This means that there exist constants $C_j$ such that
\begin{equation*}
    C_j^{-1}\|u\|_{L^2}\le\|\Si_j u\|_{L^2}\le C_j\|u\|_{L^2} , u\in \Ac(\DD).
\end{equation*}
In other words, the mapping  $\Si_j$ is a Banach isomorphism, but not an isometry, of the standard Bergman space $\Ac(\DD)$ onto the true poly-Bergman space $\Ac_{(j+1)}(\DD)$.
\begin{remark}
It might be tempting to use the isometry of Bergman spaces $W:u(z)\mapsto (1-z)^{-1} u(i\frac{z+1}{1-z})$, $W:\Ac(\DD)\to\Ac(\Pi)$, generated by the standard M\"obius transform, to derive unitary isomorphisms in the differential form between true poly-Bergman spaces on the disk from the ones found for the half plane, instead of Banach isomorphisms. This idea is, unfortunately, doomed since true poly-Bergman spaces of order higher than 1 are not invariant under conformal mappings, as simple examples (say, $u(z)=|z|^2$) show (one more drawback of polyanalytic spaces).
\end{remark}

\section{Toeplitz operators with distributional symbols}\label{Sect:ToeplitzDistr}
The standard definition of a Toeplitz operator in a Bergman type space $\Bb^2(\O)\subset L^2(\O)$ as $u\mapsto \Pb Fu$, where $\Pb $ is the orthogonal projection from $L^2(\O)$ to $\Bb^2(\O)$, works nicely  for bounded functions $F$, since $F u$ may be outside $L^2(\O)$ for an unbounded $F$. This restriction can be somewhat relaxed, by admitting \emph{some} classes of functions $F\in L^2(\O)$, since functions in  Bergman type spaces are infinitely differentiable and, therefore, locally bounded. However further extension of \emph{this} definition to more singular objects $F$ acting as symbols encounters serious obstacles.

In  \cite{RV1,RV2}, an approach has been developed to defining Toeplitz operators with fairly singular symbols. This approach is based upon a representation of  usual Toeplitz operators by means of  sesquilinear forms. If, initially, $F$ is a bounded function on $\O$, the sesquilinear form of the Toeplitz operator $\Tb_F=\Pb F$ in $\Bb^2$ can be transformed to
\begin{equation}\label{ToeplFormGen}
    \langle \Tb_F u,v\rangle=\langle \Pb F u,v\rangle= \langle  F u, \Pb v\rangle = \langle F u,v\rangle, \quad u,v\in\Bb^2.
\end{equation}
The right-hand side in \eqref{ToeplFormGen} may make sense and define a bounded sesquilinear form for rather general objects acting as $F$. Many examples of such objects are  presented in \cite{RV1,RV2,RV4}.

\subsection{Toeplitz operators in Bergman spaces on the disk and on the half-plane with symbols being derivatives of $k$--C measures.}
In $L_2(\O,dA)$, $\O=\DD$ or $\O=\Pi$, the $L^2$ scalar product is consistent, up to the complex conjugation, with the standard action of a distribution-function or distribution-measure on a smooth function on $\O$.
 Therefore,  the Toeplitz operator in $\Ac(\O)$ with distributional symbol $F\in\Dc'(\O)$ is determined by the  sesquilinear form on the right-hand side in
\begin{equation}\label{kC2a}
\fb_{F}[u,v]=\langle Fu,v\rangle_{L^2(\O,dA)}:=(F,u\bar{v})
\end{equation}
(the expression on the right-hand side in \eqref{kC2a} is understood as the action of the distribution $F$ on the smooth and, actually, real-analytic, function $u(z)\bar{v}(z)$).
This sesquilinear form is defined initially on the linear space of such functions $u,v\in\Ac$, for which the expression in \eqref{kC2a} is finite. In particular,  in the cases we consider here, the distributions in question admit a natural extension from $\Dc(\O)$ (there are no nontrivial real analytic functions in $\Dc(\O)$) to finite linear combinations of functions in the standard bases. Further, \emph{if} the sesquilinear form  \eqref{kC2a} turns out to be bounded in the norm of $\Ac(\O)$ for such functions, it is extended  to the whole of $\Ac(\O)$ by $L^2-$ continuity. If it the case, the action of the Toeplitz operator with this distributional symbol is
 \begin{equation*}
    (\Tb_F u)(z)=\fb_{F}[u, \k(z,\cdot)],
 \end{equation*}
 where, recall, $\k(z,\cdot)$ is the reproducing kernel for the space $\Ac(\O)$.

 Here we describe  briefly the realization of this approach as applied to a wide class of symbols-distributions, being  derivatives of $k$--Carleson measures on the disk or on the half-plane. The details of this realization differ somewhat from the case of the Fock space which will be explained later on.

\begin{definition}
A measure $\n$ on  $\O$ is called $k$--Carleson measure for derivatives (shortly, $k$--C measure) \emph{for the Bergman space} $\Ac(\O)$, $\O$ being $\DD$ or $\P$,
if
\begin{equation*}
  \fb_{\n,\k}[u] = \int_\O |\partial^k u(z)|^2 d\n(z)\le C \|u\|^2_{L^2(\O,dA)}
\end{equation*}
for any function $u\in \Ac(\O)$.\\
\end{definition}
In  \cite{RV2,RV4},  for the Bergman spaces on $\DD$ and $\P$, conditions, sufficient for a measure to be a $k$--C measure,  have been found. (These conditions are even necessary for the case of a positive measure.) Under these conditions, extended in a natural way to half-integer values of $k$,  the sesquilinear form
\begin{equation}\label{kC2}
    \fb_{\n,\a,\b}[u,v]=\int_{\O}\partial^{\a} u\overline{\partial^\b v}d\n\equiv(-1)^{\a+\b}(\partial^\a\bar{\partial}^\b\n, u\bar{v}), \quad u,v\in\Ac(\O), \, \a+\b=2k,
\end{equation}
is bounded in the corresponding Bergman type space $\Ac(\O)$, where the derivative of a measure is understood in the sense of distributions.

With this in view, for $\a,\b\ge0$, we define the Toeplitz operator in $\Ac(\O)$ with symbol $\partial^\a\bar{\partial}^\b \n$ as determined by the sesquilinear form $(-1)^{\a+\b}\fb_{\n,\a,\b}[u,v]$  in \eqref{kC2}, provided that the latter form is bounded in $\Ac(\O)$.

In the above cited papers, some sharp estimates have been established, involving, in particular, the dependence of the norm of the corresponding operator on the differentiation orders $\a,\b$. Here we do not need such detalization, so we  present simplified versions. Below, $B(z,r)$ denotes the disk with center $z$ and radius $r$ (this $r$ can be chosen arbitrarily) and $|\n|$ denotes the variation of the measure $\n$.

\begin{theorem}[{\cite[Theorem 6.1]{RV2}}]\label{EstBergD} Let $\n$ be a measure on $\DD$ and suppose that for some integer or half-integer $k\ge0$,
\begin{equation}\label{CondD}
    C_{k}(\n,\DD)=\sup_{z\in\DD}\left\{(1-|z|)^{-2(k+1)}|\n|(B(z, {\textstyle \frac12}(1-|z|)))\right\}<\infty.
\end{equation}
Then, for any $u,v\in\Ac(\DD)$ and any $\a,\b:\a+\b=2k$, the inequality
\begin{equation*}
|\fb_{\n,\a,\b}[u,v]|\equiv|(\partial^\a\bar{\partial}^\b\n, u\bar{v})|\le c_k C_k(\n,\DD)\|u\|_{\Ac(\DD)}\|v\|_{\Ac(\DD)}.
\end{equation*}
holds, with constant $c_k$ depending on $k$ only, i.e., the measure $\n$ is a $k$-C measure on $\DD$
\end{theorem}

The result for the Bergman space on $\Pi$ is the following (see \cite{RV4}).
\begin{theorem}\label{EstPi} Let $\n$ be a measure on $\Pi$ and suppose that for some integer or half-integer $k\ge0$,
\begin{equation}\label{CondBergPi}
    C_{k}(\n,\Pi)=\sup_{z\in\Pi}\left\{(\im z)^{-2(k+1)}|\n|(B(z, {\textstyle \frac12}\im z))\right\}<\infty.
\end{equation}
Then, for any $u,v\in\Ac(\Pi)$ and any $\a,\b:\a+\b=2k$, the inequality
\begin{equation*}
|\fb_{\n,\a,\b}[u,v]|\equiv|(\partial^\a\bar{\partial}^\b\n, u\bar{v})|\le c_k C_k(\n,\Pi)\|u\|_{\Ac(\Pi)}\|v\|_{\Ac(\Pi)}.
\end{equation*}
holds, with constant $c_k$ depending on $k$ only,
i.e., the measure $\n$ is a $k$-C measure on $\Pi.$
\end{theorem}

A bounded sesquilinear form in a Hilbert space generates a bounded operator.
\begin{theorem}\label{OperBddBerg}
Under  conditions \eqref{CondD}, resp., \eqref{CondBergPi}, the Toeplitz operator with distributional symbol $\partial^\a\bar{\partial}^\b\n$, $\a+\b=2k$ is a bounded operator in $\Ac(\DD)$, resp., $\Ac(\Pi)$.
\end{theorem}

As usual, the results on the boundedness of certain classes of Toeplitz operators are accompanied by analogous results on the compactness conditions, proved in a standard way.
\begin{theorem}\label{OperCompBerg} a. Let $\n$ be a measure on $\DD$ and
\begin{equation}\label{CompBerg}
   \lim_{\ve \to 0}\, \sup_{z\in\DD, |z|>1-\ve}\left\{(1-|z|)^{-2(k+1)}|\n|(B(z, {\textstyle \frac12}(1-|z|)))\right\} = 0.
\end{equation}
Then the Toeplitz operator in $\Ac(\DD)$ with distributional symbol   $\partial^\a\bar{\partial}^\b\n$, $\a+\b=2k$, is compact.\\
b. Let $\n$ be a measure on $\Pi$ and
\begin{equation} \label{CompBergP}
  \lim_{\ve \to 0}\,   \sup_{z\in\Pi, \im z<\ve}\left\{(\im z)^{-2(k+1)}|\n|(B(z, {\textstyle \frac12}\im z))\right\} = 0.
\end{equation}
Then the Toeplitz operator in $\Ac(\Pi)$ with distributional symbol   $\partial^\a\bar{\partial}^\b\n$, $\a+\b=2k$, is compact.
\end{theorem}
Measures satisfying \eqref{CompBerg}, \eqref{CompBergP} are called vanishing $k$--C measures in \cite{RV2}, \cite{RV4}.

\subsection{Toeplitz operators in $\Fc$ with symbols being coderivatives of $k$--FC measures.}
\begin{definition}Let $\n$ be a measure on $\C\equiv\R^2$. The measure $\n$ is called a Fock-Carleson measure for derivatives of order $k$ (shortly, a $k$-FC measure) if
\begin{equation*}
    \int_{\C}|\partial^k u|^2\o(z)d|\n|(z)\le \|u\|^2_{\Fc}, u\in\Fc.
\end{equation*}
\end{definition}
In \cite[Theorem 5.4]{RV1} sufficient conditions for a measure on $\C$ to be a $k$-FC measure, which are also necessary for positive measures, have been found. These conditions enable one to extend the notion of $k$-C measure to half-integer values of $k$.

We define a \emph{coderivative} $\pp  F$ of a distribution $F\in\Dc'(\C)$ as
\begin{equation}\label{Coder}
    \pp F=\o(z)^{-1}\partial(\o(z)F)=\left(\partial-{\textstyle\frac12}
    \bar{z}\right)F.
\end{equation}

With this definition,
\begin{equation}\label{coder1}
    (\partial  (\o F), \f)=\langle\pp F,\bar{\f}\rangle_{L^2(\C,d\r)},
\end{equation}
in the case both sides of \eqref{coder1} make sense. Thus, the coderivative of a distribution defined in \eqref{Coder} is consistent with the usual scalar product in the weighted space.

The Toeplitz operator with coderivative of a $k$-FC measure $\n$ is defined by the sesquilinear form
\begin{equation}\label{DefFock}
    \fb_{\a,\b,\n}[u,v]\equiv (-1)^{\a+\b}(\o\pp^{\a}\bar{\pp}^{\b}\n,u\bar{v})\equiv(\o \n, \partial^\a u\overline{\partial^\b v})\equiv\int_{\C}\partial^\a u(z)\overline{\partial^\b v(z)}\o(z)d\n(z).
\end{equation}

A sufficient condition for  boundedness of the sesquilinear form \eqref{DefFock} in $\Fc$ has been found in \cite{RV1}:
\begin{theorem}[{\cite[Theorem 5.4]{RV1}}]\label{ThEstFock} Let, for some integer or half-integer $k\ge0$ the measure $\n$ on $\C\equiv\R^2$ satisfies the condition
\begin{equation}\label{condFock}
 C_k(\n,\C)=\sup_{z\in\C}\{(1+|z^2|)^k|\n|(B(z,r))\}  <\infty.
\end{equation}
Then for any $u,v\in \Fc$ and $\a,\b:\a+\b=2k$, the inequality
\begin{equation*}
 |\fb_{\n,\a,\b}[u,v]|  \le c_k C_k(\n)\|u\|_{\Fc}\|v\|_{\Fc}
\end{equation*}
holds, with constant $c_k$ depending on $k$ only.
\end{theorem}

 Thus,  under the condition \eqref{condFock}, the Toeplitz operator $\Tb:u\mapsto \fb_{\n,\a,\b}[u,\k(z,\cdot)]$, is bounded in $\Fc$.
 In a similar way to Theorem \ref{OperCompBerg} the condition for compactness is found:

 \begin{theorem}\label{ThCompFock}
 Suppose that the measure $\n$ on $\C$ satisfies
 \begin{equation*}
    \lim_{R\to\infty}\sup_{z\in\C, |z|>R}\{(1+|z^2|)^k|\n|(B(z,r))\}=0.
 \end{equation*}
 Then the operator $\Tb_{\n,\a,\b}$ is compact.
 \end{theorem}

 The iterated coderivative $\pp$  in \eqref{Coder} can be expressed via the usual derivative and vice versa.

 \begin{proposition}\label{PropCoder}Let $F\in\Sc'(\C)$ be a distribution. Then, for some polynomials $\qb_k(s,t)$   and $\tilde{\qb}_k(s,t)$ of degree $k$,
 \begin{equation}\label{Codertrans1}
    \pp^k  =\qb_k(\partial, z)
 \end{equation}
 and
 \begin{equation}\label{Codertrans2}
    \partial^k =\tilde{\qb}_k(\pp,z).
 \end{equation}
 \end{proposition}
\begin{proof} The relation \eqref{Codertrans1} follows by iteration of \eqref{Coder}; the inverse, \eqref{Codertrans2}, is obtained by solving the triangular system of equations $\pp^j  =\qb_j(\partial, z),$ $j=1,\dots,\k$.
\end{proof}
We always mean that (this can be achieved by means of commutation)  in  each term in the polynomials $\qb_k(\partial, z), \tilde{\qb}_k(\pp,z)$, first the multiplication by a power of $z$ and then the differention are applied.

\section{Transformations of Toeplitz operators} \label{secTransform}  Having an operator $\Tb$ in the Hilbert space $\Hc$ and an isometry (or a Banach isomorphism) $\Zb: \Hc\to\Kc$ of $\Hc$ onto another Hilbert space $\Kc$, one can consider the operator $\Tb^\Zb=\Zb\Tb \Zb^*$ in the space $\Kc$ on the proper domain. Of course, a lot of properties of $\Tb$ and $\Tb^\Zb$ coincide. This simple idea can be successfully applied to Toeplitz operators in true Bergman type spaces described above. We will show in this Section that Toeplitz operators with symbol-function in true poly-Fock and in true poly-Bergman spaces on the half-plane and on the disk are unitary equivalent (or, in the latter case cosimilar) to certain  Toeplitz operators in the corresponding standard spaces with symbols obtained by applying proper elliptic differential operators to the initial symbols.

\subsection{Transformation of Toeplitz operators in the true poly-Bergman space on the half-plane.}
Let $F$ be a distribution on $\P$, defined at least on $C^{\infty}(\bar{\Pi})\cap L^1(\Pi).$  We consider the sesquilinear form
\begin{equation*}
 \tb_F[u,v] =  \langle F u,v\rangle_{L^2(\Pi,dA)}, \quad u= \Sbb_\Pi^j f\in\Ac_{(j+1)}(\Pi), \quad v=\Sbb_\P^j g\in\Ac_{(j+1)}(\P),
\end{equation*}
with $f,g$ being elements of the standard orthonormal basis
(further on in this subsection we will only use the Hilbert space $L^2(\Pi,dA)$ and therefore we will suppress the notation of the space in the scalar product.)

The following result leads to establishing a relation between Toeplitz operators in the true  poly-Bergman space $\Ac_{(j)}(\Pi)$ and the Bergman space $\Ac(\Pi)$.

\begin{proposition}\label{PropTranshalfplane}
Let $F$ be a distribution in the half-plane $\Pi$. Then
\begin{equation}\label{transPi}
    \langle F\Sbb_\P^j f,\Sbb_\P^j g\rangle=\langle (\Kb F) f,g\rangle,
\end{equation}
with $\Sbb$ defined in \eqref{TransP} where $\Kb$ is a differential operator of order $2j$ having the form
\begin{equation}\label{TRanshalfplane}
    \Kb=\underline{\Kb}(\Delta( y^2\cdot),\bar{\partial}(y\cdot),\partial (y\cdot)),
\end{equation}
with $\underline{\Kb}$ being a polynomial of degree $j.$ Moreover, if we assign the weight $-1$ to the differentiation and the weight $1$ to the multiplication by $y$, with weights adding under the multiplication, then all monomials in $\underline{\Kb}$ have weight $0$.
\end{proposition}
\begin{proof}We demonstrate the reasoning for the case $j=1$. The general case uses the same machinery with some tedious bookkeeping.
We consider the sesquilinear form $\fb_{F}[f,g]=\langle F\partial (y f),\partial (y g)\rangle\equiv (F,\partial (y f)\overline{\partial (y g)} )$ for $f,g$ being some elements in the standard orthonormal basis in $\Ac(\Pi)$. Due to $\partial^*=-\bar{\partial}$:
\begin{gather}\label{transHalfplane2}
    (F,\partial (y f)\overline{\partial (y g)} ) = (F,(-if+ y\partial f)\overline{(-ig+y\partial g)})\\ \nonumber
    =(F,f\bar{g})+(F,-ify\bar{\partial g})+(F,i\bar{g}y\partial f )+(F,y^2(\partial f) {\bar{\partial} \bar{g}})
    \end{gather}
   (the last transformation uses $\bar{\partial}g=0$). For the second term on the right-hand side in \eqref{transHalfplane2}, by the general rules of manipulation with distributions, we have
\begin{equation*}
 (F,-ify\bar{\partial g})= (-iyF, f\overline{{\partial g}})=(-iyF,\overline{{\partial}(\bar{f}g)})-(-iyF,(\bar{\partial}f)g )
 = (\bar{\partial}(iyF),f\bar{g}),
\end{equation*}
because $\bar{\partial}f=0$. The third term on the right in  \eqref{transHalfplane2} is transformed in a similar way, and for the last one,
 \begin{gather*}
   (F,y^2\partial f{\bar{\partial} \bar{g}} )=(y^2F,  \bar{\partial}(\partial f \bar{g})-\bar{\partial}(\partial f)\bar{g} ) =-(\bar{\partial}(y^2 F),(\partial f) \bar{g} )\\ \nonumber
   =-(\bar{\partial}(y^2 F), \partial(f \bar{g})-\{f \partial\bar{g})\})=(\partial\bar{\partial}(y^2 F), f \bar{g}),
  \end{gather*}
   again, the terms in curly bracket vanishing due to  $\bar{\partial}g=0$. Collecting the terms in \eqref{transHalfplane2}, after simple transformations, we obtain the required relation.

    Note that we would arrive  at the same result, and even somewhat faster, by making formal commutations of terms in the expression $\langle F\partial (y f),\partial (y g)\rangle$, using the commutation relations
 $[F,\partial]=-(\partial F)$,   and $[[F,\partial],\bar{\partial}]F =\frac14\D F,$
obtain the representation \eqref{TRanshalfplane}.

For higher order, the procedure of transformation is  similar, by means of formally commuting   $F$ and factors in the creation operators $\Sbb_\P^j$ in the expression $\langle F\Sbb_\P^j f,\Sbb_\P^j g\rangle\equiv(F, \Sbb_\P^j f\overline{\Sbb_\P^j g})$, so that the Cauchy-Riemann operator falls on the functions $f,g$, while any commutation with $F$ produces a derivative of $F$. It remains to notice that by the commuting operation of the terms in the expression on the left-hand side in \eqref{transPi} the weight of the terms does not change. Alternatively, one can make the calculations similar to the ones shown above, again by moving the Cauchy-Riemann operator to the functions $f,g$ and on $F$.
 \end{proof}
 As before, the equality \eqref{transPi} extends to the whole of $\Ac(\Pi)$, as soon as we know that the right-hand side or on the left-hand side is a bounded sesquilinear form in $\Ac(\Pi)$. Thus, the statement of Proposition \ref{PropTranshalfplane} can be formulated as the following theorem.

 \begin{theorem}\label{ThEquPi}The operators
 $\Tb_F(\Ac_{(j+1)}(\Pi))$ and $\Tb_{\Kb F}(\Ac(\Pi))$ are unitarily equivalent (up to a numerical factor) as soon as one of them is bounded. In this case, if one of these operators is compact, or belongs to a Schatten class, or  is finite rank, or zero, then the same holds for the other one.
 \end{theorem}
\subsection{Transformation of Toeplitz operators in the true poly-Bergman space on the disk.}\label{SubsectDisk}
The case of true poly-Bergman spaces on the disk is considered in the same way, with a single essential difference consisting in replacing the unitary equivalence of Toeplitz operators by their cosimilarity (see Definition \ref{DefCosimilar}). The proposition to follow is analogous to Proposition \ref{PropTranshalfplane}.
\begin{proposition}\label{propTransDisk} Let $F$ be a distribution in the unit circle $\DD$, defined at least on the functions in $C^{\infty}(\bar{\DD})$ Then for functions $f,g$, being linear combinations of the standard basis functions in $\Ac(\DD)$,
\begin{equation}\label{TransDD}
    \fb_F[f,g]\equiv(F\Si_j f,\overline{\Si_j g})\equiv( F,\Si_j f\overline{\Si_j g})=( (\Mb_j F)f,g),
\end{equation}
where $\Si_j$ is the operator \eqref{SDD},  $\Mb$ is an elliptic differential operator of degree $2j-2$, $\Mb=\underline{\Mb}( \D(1-|z|^2)^2\cdot, \partial z(1-|z|^2)\cdot, \bar{z}\bar{\partial}(1-|z|^2)\cdot)$,
with a polynomial $\underline{\Mb}$ of degree $j$.
In this polynomial, if we assign the weight $-1$ to differentiation, the weight $1$ to $(1-|z|^2)$ and the weight $0$ to $z,\bar{z}$, with weights adding under multiplication, all monomials have weight $0$.
\end{proposition}
\begin{proof}We will suppress ${L^2(\DD,dA)}$ in the notation of the scalar product.  Again, we demonstrate \eqref{TransDD} for $j=2$ and then indicate how the general case is handled.
We perform  calculations, similar to the ones in   \eqref{transHalfplane2}:
\begin{gather}\label{TransDD1}
    \langle F\Si_1 f,\Si_1 g\rangle = C( F, \partial(1-|z|^2)f \overline{\partial(1-|z|^2)g})\\ \nonumber =C(F, \partial((1-|z|^2)f \overline{\partial(1-|z|^2)g}))-
    C(F, (1-|z|^2)f\partial\overline{\partial(1-|z|^2)g})=\\
     \nonumber -C(( 1-|z|^2)\overline{\partial} F,f \overline{\partial(1-|z|^2)g})-((1-|z|^2) F,f\overline{\partial}{\partial(1-|z|^2)\bar{g}}).
    \end{gather}
    For the first term in \eqref{TransDD1}, we have
   \begin{gather}\label{TransDD2}\nonumber -(( 1-|z|^2)\partial F,f \overline{\partial}((1-|z|^2)\bar{g}))=-(( 1-|z|^2)\partial F, \bar{\partial}(f(1-|z|^2)\bar{g}) +(\bar{\partial}f)(1-|z|^2)\bar{g})  \\ =((1-|z|^2)\bar{\partial}( 1-|z|^2)\partial F, f\bar{g})
   ={\textstyle\frac14}(1-|z|^2)^2\D F- (1-|z|^2)z\bar{\partial} F, f\bar{g}).
   \end{gather}
    A similar transformation takes care of the second term in \eqref{TransDD1}.
    After collecting all terms, and moving all differentiations to the left, we obtain \eqref{TransDD}.  For a higher order, one should perform analogous commutations of  the terms in the operator $\Si_j$, noticing that the weight of terms does not change under these commutation.
\end{proof}
Again, the equality \eqref{TransDD} extends to all functions $f,g\in \Ac(\DD)$ as soon as we know that the right-hand side or the left-hand side is a bounded sesquilinear form in $\Ac(\DD)$.

We introduce here the notion of \emph{cosimilar} operators.
\begin{definition}\label{DefCosimilar}Let $\Hc, \Kc$ be Hilbert spaces. Operators $\Tb_1$ in $\Hc$ and $\Tb_2$ in $\Kc$ are called \emph{cosimilar} if there exists a bounded and invertible operator $\Zb:\Hc\to\Kc$ such that $\Tb_1=\Zb^*\Tb_2\Zb$.
\end{definition}

Thus, Proposition \ref{propTransDisk} leads to the following  relation between operators.
\begin{theorem}\label{Thcosim}The operators
 $\Tb_F(\Ac_{(j+1)}(\DD))$ and $\Tb_{\Mb F}(\Ac(\DD))$ are cosimilar as soon as one of them is bounded. In this case, if one of these operators is compact, or belongs to some Schatten class, or is finite rank, or is the zero operator,                                                                               then the same holds for the other one.
\end{theorem}
\begin{proof}
According to the definition of Toeplitz  operators by means of their sesquilinear forms, the relation \eqref{TransDD} can be written as
\begin{equation*}
    \langle \Tb_{F}(\Ac_{(j+1)}(\DD))\Si_jf,\Si_j g\rangle =\langle \Tb_{\Mb F}(\Ac(\DD))f,g\rangle,
\end{equation*}
which means
\begin{equation*}
   \Si_j^* \Tb_{F}(\Ac_{(j+1)}(\DD))\Si_j= \Tb_{\Mb F}(\Ac(\DD)), \, \Tb_{F}(\Ac_{(j+1)}(\DD))=(\Si_j^{-1})^* \Tb_{\Mb F}(\Ac(\DD))\Si_j^{-1}.
\end{equation*}
The operator $\Si_j^{-1}$ is a Banach isomorphism,
therefore, the operators in question are cosimilar, with $\Si_j^{-1}$ playing the role of $\Zb.$
  The remaining statements of Theorem follow automatically, since the properties in question are invariant under the multiplication by a bounded invertible operator.
\end{proof}

\subsection{Transformation of Toeplitz operators in the true poly-Fock space.}
We start by recalling Lemma 9.2 in \cite{BPR}.
\begin{theorem}\label{LemmaBPR} Consider the orthonormal in $L^2(\C)$ basis in the lowest Landau subspaces $\Lc_0$, $\vf_k=\frac{z^k}{\sqrt{k!}} \s(z)$ (recall that $\s(z)=\pi^{-1/2}e^{-|z|^2/2}$). Let $F$ be a distribution in $\Sc'(\C)$. Then for any basis functions, $\vf_k,\vf_l$
\begin{equation}\label{BPR1}
    \langle F (Q^+)^q \vf_k,(Q^+)^q \vf_l \rangle_{L^2(\C,dA)}=\langle (\Db_q F)\vf_k,\vf_l \rangle_{L^2(\C,dA)},
\end{equation}
where $\Db_q=\Db_q(\Delta)$ is a constant coefficients differential operator (understood as acting in $\Sc'$), being a polynomial of degree $q$ with positive coefficients of the Laplacian in $\C\equiv\R^2$.
\end{theorem}
Note that in \cite{BPR}, the objects in \eqref{BPR1} are defined consistently with our definition in Section \ref{Sect:ToeplitzDistr}; say,
\begin{equation*}
\langle F (Q^+)^q \vf_k,(Q^+)^q \vf_l \rangle_{L^2(\C,dA)}:=(F, (Q^+)^q \vf_k \overline{(Q^+)^q \vf_l}).
\end{equation*}
\medskip
In \cite{BPR} explicit formulas for coefficients of the operator $\Db_q(\Delta)$ have been derived, however they are not needed in our study. A simple extension of this result is, however, important. The polynomial $\Db_q$ does not depend on the choice of the basis functions in \eqref{BPR1}. Therefore, \eqref{BPR1} extends by linearity to finite linear combinations of basis functions and further, if by some reason, for a certain $F$, we happen to know that
\begin{equation*}
 |\langle (\Db_q F)\f,\psi \rangle_{L^2(\C,dA)} |\le C \|\f\|_{L^2(\C,dA)}\|\psi\|_{L^2(\C,dA)}
\end{equation*}
for $\f,\psi$ being such linear combinations, the equality \eqref{BPR1} extends by continuity to the whole of $\Lc_0$.

Now we use the surjective isometry $\Vb:L^2(\C,d\r)\to L^2(\C,dA)$, $\Vb f=\s(z)f,$ which, as it is explained in Section \ref{subsec.Fock},
relates the creation-annihilation structure in true poly-Fock spaces with the one in Landau subspaces, see \eqref{CreAnniFock}.
Setting $\f=\Vb f\in\Lc_0$, $\psi=\Vb g\in\Lc_0$ for $f,g\in \Fc$, we obtain
\begin{gather*}
  \nonumber\langle F (\Sbb^+)^q f, (\Sbb^+)^q g\rangle_{L^2(\C,d\r)}= \langle   \Vb F \Vb ^{-1} \Vb (\Sbb^+)^q f,\Vb (\Sbb^+)^q g\rangle_{L^2(\C,dA)}=(F,\Vb (\Sbb^+)^q f\overline{\Vb (\Sbb^+)^q g } )\\=(F,\Vb(\Sbb^+)^q\Vb^{-1} \f \overline{\Vb (\Sbb^+)^q \Vb^{-1}\psi}) =(F,(Q^+)^q\f\overline{(Q^+)^q\psi} )=
  ( (\Dc_q F), \Vb f \bar{\Vb g}).
\end{gather*}
For $F\in \Sc'$, the above equality holds for every  $\f=\Vb f$, $\psi=\Vb g$ with $f,g$ being elements in the standard monomial basis in $\Fc$, and, therefore it extends to the finite linear combinations of such elements. Suppose that we know that the inequality
\begin{equation}\label{bddnessFc}
|( (\Db_q F),\Vb f\overline{\Vb g})|\le C\|f\|_{\Fc}\|g\|_{\Fc}
\end{equation}
is satisfied for elements in this subspace. Since $\Qb^+$ is (up to a constant) a surjective isometry of Landau subspaces,
\eqref{bddnessFc} means that
\begin{equation}\label{bddnessFc1}
  |\fb_F[u,v]|\equiv |\langle Fu,v\rangle_{L^2(\C,d\r)}|\le C\|u\|_{L^2(\C,d\r)}\|v\|_{L^2(\C,d\r)}, \ u=(\Qb^+)^qf\in\Fc_{(q+1)}, \ v=(\Qb^+)^q g\in\Fc_{(q+1)}
\end{equation}
for a dense linear subspace of $f,g$ in $\Fc$, and, therefore, extends to the whole of $\Fc$.
Thus we arrive to the following transformation result.
\begin{proposition}\label{propTransFc}
Let for some $F\in \Sc'(\C)$ the inequality
\eqref{bddnessFc1} be satisfied for all $u,v$ in a dense linear subset of linear combinations of basic functions in the true poly-Fock space $\Fc_{(q+1)}$. Then
\begin{equation*}
    \langle F(\Sbb^+)^q f, (\Sbb^+)^q g\rangle_{\Fc_{(q+1)}}=\langle \Vb^{-1}(\Db_q F)\Vb f,g\rangle_{\Fc}=(\Db_q F, \Vb f\Vb \bar{g}),
\end{equation*}
where $f=\Qb^q u\in\Fc$ and $g=\Qb^q v\in\Fc$.
\end{proposition}

\section{Properties of Toeplitz operators in true poly-Bergman type spaces}
Now we use our transformation  results   to investigate some properties of  operators in true spaces of polyanalytic functions. Since the formulations and explanations are mostly the same for all three types of spaces we consider in this paper, we will refer to all of them as to true poly-Bergman type spaces (\emph{tpBt} spaces.)
\subsection{Compactness and degeneracy}
\subsubsection{Symbols with compact support and generalizations} We start with the most simple results. It was established in \cite{RV1,RV2,RV4} that a Toeplitz operator with symbol-distribution having compact support in $\C$, $\DD$, $\Pi$ is compact in the corresponding Bergman type space. Of course, if the symbol $F$ belongs to $\Es'(\O)$, $\O=\C,\DD$ or $\Pi$, the result of applying a differential operator to $F$ has compact support as well. This leads to the first, quite simple, result.
\begin{theorem}\label{ThCompSupp} If the symbol $F$ has compact support in $\O$ then the Toeplitz operator with symbol $F$ in the corresponding \emph{tpBt}-space is compact.
\end{theorem}

Now, as it will happen repeatedly further on, it is quite possible that the symbol $F$, probably, does not have compact support, but, after applying to $F$ a differential operator $\Eb$ (this is $\Kb,\Mb,\Db$, depending on the space), we obtain the symbol $\Eb F$ with compact support. Thus, we have the following theorem.

\begin{theorem}\label{AdvThComSup} Suppose that for  a symbol $F$ the symbol  $\Eb F$ has compact support. Then the Toeplitz operator with symbol $F$ in the corresponding \emph{tpBt} space is compact.
\end{theorem}

\begin{remark}\label{RemCompSup} As explained in \cite[Section 5]{RV2}, a Toeplitz operator with a compactly supported distributional symbol in $\Ac(\DD)$ has eigenvalues (and $s$-numbers, if the operator is not self-adjoint) tending to zero at least exponentially. The same holds for operators in the space $\Ac(\Pi)$; for the Fock space, the eigenvalues ($s$-numbers) decay superexponentially. By our transformation theorems, the same takes place for Toeplitz operators with compactly supported symbol, or even symbol with compactly supported $\Eb F$, in the corresponding \emph{tpBt} space. We are going to explore this topic in later publications.
\end{remark}

\subsubsection{Uniqueness}
It is well known that a Toeplitz operator in Bergman type spaces with more or less non-crazy symbol cannot be zero unless the symbol is zero. For example, if a radial symbol in the Fock space grows at infinity not faster than $\exp{a|z^2|}$, $a<1$, such uniqueness takes  place (examples of non-uniqueness, therefore, involve symbols growing at infinity \emph{very} rapidly, see \cite{GV}). The same is correct for operators in Bergman spaces on the disk and on the half-plane. To break uniqueness, the symbol should grow rapidly and oscillate very  fast near the boundary.
For operators in \emph{tpBt} spaces the situation is completely different.

Suppose that $F$ is a symbol in $\O$ satisfying in $\O$ the equation $\Eb F=0$ in the sense of distributions, where, as before  $\Eb$ is $\Kb,\Mb$ or $\Db$, depending on the space in question. Then, by the results of Section \ref{SF}, the Toeplitz operator with symbol $F$ in the corresponding \emph{tpBt} space, will be the zero operator. We delay to another publication the detailed analysis of the properties of such equations. We note a simple, however not quite trivial, special case.

\begin{theorem}Let $F$ be a bounded function (of course, this restriction can be relaxed) with compact support in $\O$. Suppose that the Toeplitz operator in a tpBt space with symbol $F$ is zero. Then $F$ is zero.
\end{theorem}
Not giving a formal proof, we just explain that the operator with symbol $\Eb F$ in the usual Bergman space is zero, therefore, $\Eb f=0$. A solution of this equation is a real-analytic function and therefore cannot have compact support unless it is zero everywhere.

\subsubsection{The finite rank property} Recently there has been certain research activity concerning the problem of Toeplitz operators in Bergman type spaces having finite rank. Some final results on this topic have been obtained in \cite{AR}, \cite{RS} (see also the bibliography there.)
Due to these results, a Toeplitz operator in $\Ac(\DD)$ and $\Fc$ with distributional symbol $G$, subject to some restrictions concerning 'growth at infinity', may have finite rank if only if the symbol $G$ is, in fact, a finite sum of point masses and their derivatives. In particular, this characterization holds for symbols with compact support.  More generally, if a distribution $G$ satisfies these growth conditions then any derivative of $G$ satisfies them (see \cite{RS} for details). From these results, it follows, that no function with a moderate growth (and, for sure, no function  with compact support) may be the symbol of a finite rank Toeplitz operator in a Bergman type space.

Now we touch upon this problem for operators in a \emph{tpBt} spaces. Suppose that $F$ is a function satisfying the conditions mentioned above. Suppose that the Toeplitz operator in the \emph{tpBt} space has finite rank. Then, by our transformation results, the Toeplitz operator in the corresponding usual Bergman type space with \emph{distributional} symbol $G = \Eb F$  has finite rank as well, therefore, $G$ must  be a finite linear combination of point masses and their derivatives. Thus any solution $F$ of the elliptic equation  $\Eb F=G$ produces a Toeplitz operator with finite rank. In particular, if the order $j+1$ of the \emph{tpBt} space is larger than the singularity order of the distribution $G$ then, by the elliptic regularity, any solution of the equation $\Eb F=G$ is a continuous function. Thus, in   \emph{tpBt} spaces there can exist symbols-functions producing nontrivial finite rank Toeplitz operators. Note, however, that if we suppose that $F=0$ on some open set in $\O$ (in particular, if it has compact support in $\O$), then the function $F$, being a solution of a real-analytic elliptic equation $\Eb F=G$, must be real-analytic everywhere outside the singularities of $G$, and this means, outside some finite set. Therefore, $F$ must be zero outside this set, where $F$ must be a finite linear combination of point masses and their derivatives.

\subsection{Boundedness}
Here we combine the results on transformation of Section \ref{secTransform} and the boundedness conditions described in Section \ref{Sect:ToeplitzDistr}.
We have everything at hand already.
\begin{theorem}Let $F$ be a function $\DD$ such that
\begin{equation}\label{BddOpDisk}
\sup_{z\in\DD}\bigg\{(1 - |z|)^{-2} \int_{B(z, \frac12(1-|z|))}|F(w)|dA(w)\bigg\}<\infty
\end{equation}
then the Toeplitz operator with symbol $F$ in $\Ac_{(k+1)}$ is bounded. Under the additional condition
\begin{equation}\label{CompOpDisk}
   \lim_{\ve\to0} \sup_{|z|>1-\ve}\bigg\{(1 - |z|)^{-2} \int_{B(z, \frac12(1-|w|))}|F(w)|dA(w)\bigg\}=0
\end{equation}
this operator is compact.
\end{theorem}
Thus, strangely enough, the boundedness conditions for Toeplitz operators in $\Ac_{(j+1)}(\DD)$ are the same as in $\Ac(\DD)$.
 \begin{proof} By Theorem \ref{Thcosim}, it is sufficient to prove that the Toeplitz operator in the Bergman space $\Ac(\DD)$ with distributional symbol $\Mb_j F$ is bounded. We look closer at the structure of the operator $\Mb_j$. By Proposition \ref{propTransDisk}, it consists of the sum of terms of weight zero.  These terms can be transformed to  the ones $h_\a(z)\ell_\a(\partial, \bar{\partial})(1-|z|^2)^{|\a|}$, with $|\a|\le 2j$. Here, $\ell_\a(\partial, \bar{\partial})$ is a constant coefficients differential operator of order $|\a|$ and $h_{\a}$ is a smooth bounded function on $\DD$. Each such term generates a bounded Toeplitz operator in $\Ac(\DD)$, by Theorem \ref{EstBergD}, applied to the measure $\n_\a$, $\n_a(E)=\int_E(1-|w|^2)^{|\a|}|F(w)|dA(w)$ (the factors depending on  the distance to the boundary cancel). The statement about compactness follows in the same way from Theorem \ref{OperCompBerg}.
 \end{proof}

The boundedness and compactness theorem for Toeplitz operators in $\Ac_{(j+1)}(\Pi)$ follows completely analogously from Theorem \ref{ThEquPi}, Proposition \ref{PropTranshalfplane} and Theorem \ref{EstPi}.

\begin{theorem}\label{ThBddPpolyPi}
Let $F$ be a function on $\Pi$ If
\begin{equation}\label{BddOPerPi}
    \sup_{z\in\Pi} \bigg\{ (\im{z})^{-2}\int_{B(z, \frac12\im{z})}|F(w)|dA(w)\bigg\}<\infty
\end{equation}
then the Toeplitz operator in $\Ac_{(j)}$ with symbol $F$ is bounded. If, additionally,
 \begin{equation}\label{CompOpDDD}
   \lim_{\ve\to0} \sup_{\im{z}<\ve}\bigg\{ (\im{z})^{-2} \int_{B(z, \frac12\im{z})}|F(w)|dA(w)\bigg\}=0
\end{equation}
then this operator is compact.
\end{theorem}
Finally, we present the result for true poly-Fock spaces. It follows in the same way from Proposition \ref{propTransFc} and Theorem \ref{ThEstFock}.  Here, unlike two previous cases, the boundedness and compactness conditions are different for poly-Fock spaces of different order. This is explained by the absence of cancelation of weight factors in the transformed symbol with the weight factors in the boundedness conditions.
\begin{theorem}Suppose that the function $F$ satisfies
\begin{equation*}
    \sup_{z\in\C}\bigg\{(1+|z|^2)^{(j-1)} \int_{B(z,r)}|F(w)|dA(w)\bigg\}<\infty.
\end{equation*}
Then the Toeplitz operator in $\Fc_{(j)}$ with symbol $F$ is bounded in $\Fc_{(j)}$. If, additionally
\begin{equation*}
   \lim_{R\to\infty} \sup_{|z|>R}\bigg\{(1+|z|^2)^{(j-1)} \int _{B(z,r)}|F(w)|dA(w)\bigg\}=0,
\end{equation*}
then this operator is compact.
\end{theorem}

\begin{proof} As it is shown in Proposition \ref{PropCoder}, the derivatives entering into the operator $\Db$ can be expressed in a linear way through the coderivatives. To each of the corresponding terms in $\Db F$, we can, therefore, apply Theorem \ref{ThEstFock}, which gives the required estimate. The compactness follows automatically.
\end{proof}

\end{document}